\documentclass{lms}

\usepackage{color}
\usepackage{xcolor}

\usepackage{amsmath}
\usepackage{amssymb}
\usepackage{float}
\usepackage{times}
\usepackage{enumitem} 

\usepackage{tikz}
\tikzstyle{every picture} = [scale=.6]
\tikzstyle{every node} = [draw, fill=white, circle, inner sep=0pt, minimum size=4pt, text depth=1ex]
\tikzstyle{every label} = [draw=none, rectangle, inner sep = 2pt]
\tikzstyle{n} = [draw=none, rectangle, inner sep=0pt]

\newtheorem{theorem}{Theorem}[section] 
\newtheorem{lemma}[theorem]{Lemma}     
\newtheorem{corollary}[theorem]{Corollary}
\newtheorem{proposition}[theorem]{Proposition}
\newtheorem{mthm}{Theorem}

\newcommand{\tuple}[1]{\ensuremath{\langle{#1}\rangle}}
\newcommand{\set}[1]{\{ #1 \}}

\newcommand{\EqT}{{\rm Eq}}
\newcommand{\defiff}{\: :\Longleftrightarrow\:}

\newcommand{\f}{\varphi}
\newcommand{\p}{\psi}

\newcommand{\de}{\delta}
\newcommand{\eps}{\varepsilon}
\newcommand{\e}{{\rm e}}
\newcommand{\pd}{\cdot}
\newcommand{\jn}{\vee}
\newcommand{\mt}{\wedge}
\newcommand{\m}{\mathbf} 
\newcommand{\N}{{\mathbb N}}
\newcommand{\Z}{{\mathbb Z}}
\newcommand{\Q}{{\mathbb Q}}
\newcommand{\R}{{\mathbb R}}
\newcommand{\iv}[1]{{#1}^{-1}}

\newcommand{\De}{\mathrm{\Delta}}

\newcommand{\Si}{\mathrm{\Sigma}}

\newcommand{\lang}{{\mathcal L}}
\newcommand{\cls}[1]{\mathsf{#1}}

\newcommand{\eq}{\approx}

\newcommand{\qearrow}{\Rightarrow}

\newcommand\nbd[1]{\protect\nobreakdash#1\hspace{0pt}}

\newnumbered{assertion}{Assertion}    
\newnumbered{conjecture}{Conjecture}  
\newnumbered{definition}{Definition}
\newnumbered{hypothesis}{Hypothesis}
\newnumbered{remark}{Remark}
\newnumbered{note}{Note}
\newnumbered{observation}{Observation}
\newnumbered{problem}{Problem}
\newnumbered{question}{Question}
\newnumbered{algorithm}{Algorithm}
\newnumbered{example}{Example}
\newunnumbered{notation}{Notation} 

\simpleequations

\title[Equational theories of idempotent semifields]
 {Equational theories of idempotent semifields} 

\author{G. Metcalfe and S. Santschi}

\classno{12K10, 16Y60, 03C05, 06F15}

\extraline{Supported by Swiss National Science Foundation (SNSF) grant No. 200021\textunderscore 215157.}

\begin{document}
\maketitle


\begin{abstract}
This paper provides answers to several open problems about equational theories of idempotent semifields. In particular, it is proved that (i) no equational theory of a non-trivial class of idempotent semifields has a finite basis; (ii) there are continuum-many equational theories of classes of idempotent semifields; and (iii) the equational theory of the class of idempotent semifields is co-NP-complete. This last result is also used to determine the complexity of deciding the existence of a right order on a free group or free monoid satisfying finitely many given inequalities.
\end{abstract}


\section{Introduction}\label{s:introduction}

An {\em idempotent semiring} is an algebraic structure $\tuple{S,\jn,\pd,\e}$ satisfying
\begin{enumerate}[label=(\roman*)]
\item	$\tuple{S,\pd,\e}$ is a monoid;
\item $\tuple{S,\jn}$ is a semilattice (i.e., an idempotent commutative semigroup); and
\item $a(b\jn c)d=abd\jn acd$ for all $a,b,c,d\in S$.
\end{enumerate}
If $\tuple{S,\pd,\e}$ is the monoid reduct of a group, then $\tuple{S,\jn,\pd,\e}$ is called an {\em idempotent semifield}. Such structures arise naturally in many areas of mathematics, including  idempotent analysis, formal language theory,  tropical geometry,  and mathematical logic; for details and further references, see~\cite{Gol99,Gun98,HW95}.\footnote{Alternative definitions of an idempotent semiring (also known as a {\em dioid} or an {\em ai-semiring}) may be found in the literature that differ with respect to the presence and/or role of constant symbols in the signature. In particular, an idempotent semiring is sometimes defined without $\e$, where $\tuple{S,\pd}$ is required to be a semigroup, and sometimes with both $\e$ and a further constant symbol $0$, where $0$ is interpreted as the neutral element of $\jn$. In the latter case, the definition of an idempotent semifield is changed so that $\tuple{S{\setminus}\{0\},\pd,\e}$ is the monoid reduct of a group. As explained in Section~\ref{s:signature}, however, our results extend also to these settings.\label{f:alternatives}}  

Any idempotent semifield $\tuple{S,\jn,\pd,\e}$ expanded with the group inverse operation $\iv{}$ and lattice meet operation $\mt$ defined by setting $a\mt b:=\iv{(\iv{a}\jn\iv{b})}$ is a {\em lattice-ordered group} ({\em $\ell$\nbd{-}group}, for short): an algebraic structure $\tuple{L,\mt,\jn,\pd,\iv{},\e}$ such that $\tuple{L,\pd,\iv{},\e}$ is a group; $\tuple{L,\mt,\jn}$ is a lattice with an order defined by $x\le y\defiff x\jn y=y$; and multiplication is order-preserving, i.e., $a \le b\:\Longrightarrow\:cad \le cbd$, for all $a,b,c,d \in L$. Indeed, idempotent semifields are precisely the semiring reducts of $\ell$-groups. 

In this paper, we provide answers to several open problems about equational theories of idempotent semifields and related structures. Although these problems have been solved for $\ell$-groups, restricting to fewer operations requires the development of new proof methods and yields notably different results. 

In order to present these results, let us first  recall some basic terminology. A {\em signature} $\lang$ is a set of operation symbols with finite arities, and an {\em $\lang$-algebra} $\m{A}$ consists of a non-empty set $A$ equipped with an $n$-ary function on $A$ for each operation symbol of $\lang$ of arity $n$. An {\em $\lang$-term} is built inductively using the operation symbols of $\lang$ and a countably infinite set of variables, and the {\em $\lang$-term algebra} $\m{Tm}(\lang)$ consists of the set of $\lang$-terms equipped with the term-building operation symbols of $\lang$. An {\em $\lang$-equation} is an ordered pair of $\lang$-terms $s,t$, written $s\eq t$, and is {\em satisfied} by an $\lang$-algebra $\m{A}$, written $\m{A}\models s\eq t$, if $\f(s)=\f(t)$, for any homomorphism $\f\colon\m{Tm}(\lang)\to\m{A}$. Given any class of  $\lang$-algebras $\cls{K}$ and set of $\lang$-equations $\Si$, we denote by $\cls{K}\models\Si$ that $\m{A}\models s\eq t$ for all $\m{A}\in\cls{K}$ and $s\eq t\in\Si$. The {\em equational theory} $\EqT(\cls{K})$ of a  class of  $\lang$-algebras $\cls{K}$ is the set of all $\lang$-equations $s\eq t$ such that $\cls{K}\models s\eq t$. 

In Section~\ref{s:finite-basis-problem}, we provide a complete answer to the finite basis problem for idempotent semifields. Let $\cls{K}$ be any class of $\lang$-algebras, and call it {\em non-trivial} if at least one of its members is non-trivial, i.e., has more than one element. A {\em basis} for the equational theory of $\cls{K}$ is a set of equations $\Si\subseteq\EqT(\cls{K})$ such that every equation in $\EqT(\cls{K})$ is a logical consequence of $\Si$, that is, if $\m{A}\models\Si$ for some $\lang$-algebra $\m{A}$, then $\m{A}\models\EqT(\cls{K})$. If $\EqT(\cls{K})$ has a finite basis, then $\cls{K}$ is said to be {\em finitely based}. Notably, the equational theory of the $\ell$-group $\tuple{\Z,\min,\max,+,-,0}$ is finitely based, but this is not the case for the semifield $\tuple{\Z,\max,+,0}$ or any other totally ordered semifield~\cite[Theorem~48]{Aceto2003}. Indeed, although countably infinitely many equational theories of $\ell$-groups have a finite basis (see,~e.g.,~\cite[Section~7.2]{AF88}), we prove here that:

\newtheorem{t:basis}{Theorem~\ref{t:basis}}
\begin{t:basis*}
There is no non-trivial class of idempotent semifields that is finitely based.
\end{t:basis*}

In Section~\ref{s:cardinality-problem}, we determine the number of equational theories of classes of idempotent semifields. Using a technique of `inverse elimination' to translate between equations in the different signatures,  we obtain a one-to-one correspondence between a family of equational theories of $\ell$-groups that is known to be uncountable and equational theories of certain classes of idempotent semifields, thereby proving:

\newtheorem{t:continuum}{Theorem~\ref{t:continuum}}
\begin{t:continuum*}
There are continuum-many equational theories of classes of idempotent semifields.
\end{t:continuum*}

In Section~\ref{s:complexity}, we establish the complexity of deciding equations in the class of idempotent semifields. The equational theory of the class of $\ell$-groups is known to be co-NP-complete~\cite[Theorem~8.3]{GM16} and we prove here that this is also the case for the restricted signature, that is:

\newtheorem{t:conp}{Theorem~\ref{t:conp}}
\begin{t:conp*}
The equational theory of the class of idempotent semifields is co-NP-complete.
\end{t:conp*}

\noindent
We also use this result to show that the problem of deciding if there exists a right order on the free $n$-generated group ($n\ge 2$) whose positive cone contains a given finite set of elements is NP-complete and that the same is true for the problem of deciding if there exists a right order on the free countably infinitely generated monoid that satisfies a given finite set of inequalities.

Finally, in Section~\ref{s:signature}, we extend the results of the paper to related structures considered in the literature:  expansions of idempotent semifields with the meet operation, $\e$-free reducts of idempotent semifields, and idempotent semifields extended with a neutral element $0$ for the join operation.


\section{The finite basis problem}\label{s:finite-basis-problem}

Let $\lang_m$ and $\lang_s$ be the signatures of monoids and idempotent semirings, respectively. Following~\cite{Jac08}, let the {\em flat extension} of an $\lang_m$-algebra $\m{M} = \tuple{M,\pd,\e}$ be the $\lang_s$-algebra $\flat(\m{M})= \tuple{M\cup\set{\top},\jn,\star,\e}$, where $\top\not\in M$ and for all $a,b \in M\cup\set{\top}$,
\[
a\star b :=
\begin{cases}
a\pd b & \text{if }a,b\in M\\
\top & \text{otherwise}
\end{cases}
\qquad\text{and}\qquad
a\jn b :=
\begin{cases}
a & \text{if }a=b\\
\top & \text{otherwise.}
\end{cases}
\]
It is easily confirmed that if $\m{M}$ is a monoid, then $\tuple{M\cup\set{\top},\star,\e}$ is a monoid, $\tuple{M\cup\set{\top},\jn}$ is a semilattice of height one, and $\top$ is an absorbing element for both the binary operations of $\flat(\m{M})$. The following result provides a necessary and sufficient condition for $\flat(\m{M})$ to be an idempotent semiring. 

\begin{lemma}[cf.~{\cite[Lemma~2.2]{JRZ22}}]
Let  $\m{M}$ be any monoid. Then $\flat(\m{M})$ is an idempotent semiring if and only if $\m{M}$ is cancellative, i.e., $cad=cbd\:\Longrightarrow\:a=b$, for all $a,b,c,d\in M$.
\end{lemma}
\begin{proof}
Suppose first that $\flat(\m{M})$ is an idempotent semiring and $cad=cbd$ for some $a,b,c,d\in M$. Then $c(a\jn b)d=cad\jn cbd=cad\neq\top$, so $a\jn b\neq\top$ and $a=b$. For the converse, suppose that $\m{M}$ is cancellative and consider any $a,b,c,d\in M$. If $b=c$, then, clearly, $a(b\jn c)d=abd\jn acd$. Otherwise, $b\neq c$ and, by cancellativity, $abd\neq acd$, so  $a(b\jn c)d=a\top d =\top = abd\jn acd$.
\end{proof}

Consider now  the monoid reduct $\m{Z} = \tuple{\Z,+,0}$ of the additive integer group, and monoid reducts $\m{Z}_n = \tuple{Z_n,\pd, \e}$ of the cyclic groups of order $n\in\N^{>0}$. Since these monoids are cancellative, their flat extensions, $\flat(\m{Z})$ and $\flat(\m{Z}_n)$ ($n\in\N^{>0}$), are idempotent semirings with a flat semilattice structure, as depicted in Figure~\ref{fig:Zp-Z}. We will prove first that the equational theory of any finitely based class of idempotent semirings $\cls{K}$ is satisfied by $\flat(\m{Z})$ if and only if it is satisfied by $\flat(\m{Z}_p)$ for every prime $p$ greater than some suitably large $n\in\N$ (Corollary~\ref{c:fin-base}). We will then show that the equational theory of any non-trivial class of idempotent semifields is satisfied by $\flat(\m{Z})$, but not by $\flat(\m{Z}_n)$ for any $n\in\N^{>0}$, thereby establishing that the class is not finitely based (Theorem~\ref{t:basis}).

\begin{figure}
\begin{tikzpicture}[xscale=.8,yscale=.7]
\draw 
(-3.6,2)node{}--
(-6,5)node[label=above:$\top$]{}--
(-8.4,2)node{}
(-7.2,2)node{}--(-6,5)
(-6,2)node{}--(-6,5)
(-4.8,2)node{}--(-6,5)
(-8.65,0.4)node[n,label=above:$-2$]{}
(-7.4,0.4)node[n,label=above:$-1$]{}
(-6,0.4)node[n,label=above:$0$]{}
(-4.8,0.4)node[n,label=above:$1$]{}
(-3.6,0.4)node[n,label=above:$2$]{}
(-9.05,1.85)node[n]{$\dots$}
(-2.8,1.85)node[n]{$\dots$};
\end{tikzpicture}
\quad\quad
\begin{tikzpicture}[xscale=.8,yscale=.7]
\draw 
(-3,2)node{}--
(-6,5)node[label=above:$\top$]{}--
(-9,2)node{}
(-7.5,2)node{}--(-6,5)
(-6,2)node{}--(-6,5)
(-4.5,1.85)node[n]{$\dots$}
(-9,0.4)node[n,label=above:$\e$]{}
(-7.5,0.4)node[n,label=above:$a$]{}
(-6,0.4)node[n,label=above:$a^2$]{}
(-2.8,0.4)node[n,label=above:$a^{n-1}$]{};
\end{tikzpicture}
\centering
\caption{Hasse diagrams of $\flat(\m{Z})$ and $\flat(\m{Z}_n)$.}
\label{fig:Zp-Z}
\end{figure}

For a signature $\lang$ containing the binary operation symbol $\jn$, we call an $\lang$-equation of the form $s\jn t\eq t$, often written $s \le t$, an {\em $\lang$-inequation}. We call an $\lang_s$-inequation \emph{simple} if it is of the form $s\leq t_1\jn\cdots\jn t_n$, where $s,t_1,\dots, t_n$ are $\lang_m$-terms, and call this simple $\lang_s$-inequation {\em left-regular} if each variable occurring in $s$ occurs in at least one of $t_1,\dots,t_n$.

\begin{remark}\label{r:simplesuffice}
Clearly, an idempotent semiring satisfies an $\lang_s$-equation $s\eq t$ if and only it satisfies the $\lang_s$-inequations $s\le t$ and $t\le s$. Hence, using the distributivity of multiplication over binary joins and the fact that $a\jn b\le c \iff a\le c\text{ and }b\le c$ for all elements $a,b,c$ of an idempotent semiring, there exists an algorithm that produces for every $\lang_s$-equation $\eps$, a finite set $\Sigma$ of simple $\lang_s$-inequations such that an arbitrary idempotent semiring satisfies $\eps$ if and only if  it satisfies $\Si$.
\end{remark}

\begin{remark}\label{r:regular}
If an idempotent semiring satisfies a simple $\lang_s$-inequation that is not left-regular, then --- by substituting all variables, except for one that occurs on the left and not the right, with $\e$ --- it must also satisfy $x^n\le\e$ for some $n\in\N^{>0}$. Hence, reasoning contrapositively, if a simple $\lang_s$-inequation is satisfied by an idempotent semiring containing an element greater than $\e$, it must be left-regular. \end{remark}

For a signature $\lang$, an {\em $\lang$-quasiequation} is an ordered pair consisting of a finite set of $\lang$-equations $\Si$ and an $\lang$-equation $s\eq t$, written $\Si\qearrow s\eq t$, and is  satisfied by an $\lang$-algebra $\m{A}$ if for any homomorphism $\f\colon\m{Tm}(\lang)\to\m{A}$, whenever $\f(s')=\f(t')$ for all $s'\eq t'\in\Si$, also $\f(s)=\f(t)$. Given any simple $\lang_s$-inequation $\eps = (s \leq t_1\jn\cdots\jn t_n)$, we define a corresponding $\lang_m$-quasiequation
\[
Q(\eps) :=\set{t_1\eq t_2, \dots, t_1 \eq t_n} \qearrow t_1 \eq s.
\]

\begin{lemma}\label{l:flat-group}
Let $\m{M}$ be any monoid. Then for any left-regular simple $\lang_s$-inequation $\eps$,
\[
\flat(\m{M}) \models \eps \iff \m{M} \models Q(\eps).
\]
\end{lemma}

\begin{proof}
Let $\eps = (s \leq t_1\jn\cdots\jn t_n)$ be any left-regular simple $\lang_s$-inequation. 

For the left-to-right direction, suppose contrapositively that $\m{M} \not\models Q(\eps)$. Then $\f(t_1) =\cdots = \f(t_n)$ and $\f(t_1)\neq\f(s)$ for some homomorphism $\f\colon\m{Tm}(\lang_m)\to\m{M}$. Let $\hat{\f}\colon\m{Tm}(\lang_s)\to\flat(\m{M})$ be the homomorphism defined by setting $\hat{\f}(x):=\f(x)$ for every variable $x$. Clearly, $\hat{\f}(u)=\f(u)$ for each $\lang_m$-term $u$ and therefore $\hat{\f}(t_1\jn\cdots\jn t_n)=\hat{\f}(t_1)\jn\cdots\jn\hat{\f}(t_n)=\f(t_1)\jn\cdots\jn\f(t_n)=\f(t_1)\neq\f(s)=\hat{\f}(s)$. But $\f(t_1),\f(s)\in M$, so $\hat{\f}(t_1\jn\cdots\jn t_n)\not\le\hat{\f}(s)$. Hence $\flat(\m{M}) \not\models\eps$.

For the converse direction, suppose contrapositively that $\flat(\m{M}) \not \models \eps$. Then $\p(s) \nleq \p(t_1\jn\cdots\jn t_n)$ for some homomorphism  $\p\colon\m{Tm}(\lang_s)\to\flat(\m{M})$. It follows from the definition of the order of $\flat(\m{M})$ that $\p(t_1\jn\cdots\jn t_n)\neq\top$ and $\p(t_1)=\cdots=\p(t_n)\neq\top$. Hence, since $\eps$ is left-regular, $\p(x)\in M$ for every variable $x$ occurring in $\eps$, and there exists a homomorphism $\hat{\p}\colon\m{Tm}(\lang_m)\to\m{M}$ satisfying $\hat{\p}(x)=\p(x)$ for all such $x$. Clearly, $\hat{\p}(t_1)=\cdots=\hat{\p}(t_n)$ and $\hat{\p}(t_1) \neq\hat{\p}(s)$. So $\m{M} \not\models Q(\eps)$.
\end{proof}

\begin{remark}
The right-to-left direction of Lemma~\ref{l:flat-group} does not hold for all simple $\lang_s$-inequations that are not left-regular; e.g., $\m{Z}_2 \models \emptyset\qearrow \e \eq x^2$,  but $\flat(\m{Z}_2) \not\models x^2 \le \e$.
\end{remark}

\begin{lemma}\label{l:p-group}
Let $\De$ be any finite set of $\lang_m$-quasiequations. Then $\m{Z}\models \De$ if and only if there exists an $n\in\N$ such that $\m{Z}_p \models \De$ for every prime $p>n$.   
\end{lemma}
\begin{proof}
For the right-to-left direction, suppose that there exists an $n \in \N$ such that $\m{Z}_p \models \De$ for every $p\in P_n:=\set{p\in\N\mid \text{$p>n$ and $p$ is prime}}$. Since $\De$ is a set of  $\lang_m$-quasiequations, $\prod_{p\in P_n} \m{Z}_p\models\De$. The projection maps $\pi_p\colon\m{Z}\to\m{Z}_p$ $(p\in P_n)$ induce an embedding of $\m{Z}$ into $\prod_{p\in P_n} \m{Z}_p$, so also $\m{Z}\models \De$. 

For the converse direction, suppose that $\m{Z}\models \De$. Let $T$ be a set of first-order sentences axiomatizing the class of Abelian groups and let $\alpha$ be the conjunction of the  $\lang_m$-quasiequations in~$\De$.  Since an $\lang_m$-quasiequation is satisfied by all torsion-free Abelian groups if and only if it is satisfied by $\m{Z}$, it follows that $T\cup\set{\set{x^k\eq\e} \qearrow x\eq\e\mid k\in\N}\models\alpha$. By compactness, there exists an $n \in \N$ such that $T\cup\{\set{x^k\eq\e}\qearrow x\eq\e\mid  k\in\N\text{ and }k\le n\}\models\alpha$. Hence $\m{Z}_p \models \De$ for every prime $p>n$.   
\end{proof}

\begin{proposition}\label{p:ZtoZp}
Let $\Si$ be any finite set of left-regular simple $\lang_s$-inequations. Then $\flat(\m{Z})\models \Si$ if and only if there exists an $n\in\N$ such that $\flat(\m{Z}_p) \models \Si$ for every prime $p>n$.   
\end{proposition}
\begin{proof*}
Let $\De:=\set{Q(\eps)\mid\eps\in\Si}$. Then
\begin{align*}
 \flat(\m{Z}) \models \Si  &\iff \m{Z} \models\De & \text{(Lemma~\ref{l:flat-group})} \\
&\iff \text{there exists an $n\in \N$ such that $\m{Z}_p\models\De$ for every prime $p>n$} & \text{(Lemma~\ref{l:p-group})}\\
&\iff \text{there exists an $n\in \N$ such that $\flat(\m{Z}_p)\models \Si$ for every prime $p>n$} & \text{(Lemma~\ref{l:flat-group})}
&.&\proofbox
\end{align*}
\end{proof*}

\begin{corollary}\label{c:fin-base}
Let $\cls{K}$ be any finitely based class of idempotent semirings. Then $\flat(\m{Z}) \models \EqT(\cls{K})$ if and only if there exists an $n\in \N$ such that $\flat(\m{Z}_p) \models \EqT(\cls{K})$ for every prime $p>n$.
\end{corollary}
\begin{proof}
By assumption and Remark~\ref{r:simplesuffice}, there exists a finite basis $\Si$ for $\EqT(\cls{K})$ that consists of simple $\lang_s$-inequations. Moreover, by Remark~\ref{r:regular}, every simple $\lang_s$-inequation satisfied by $\flat(\m{M})$ for some monoid $\m{M}$ is left-regular. Hence $\flat(\m{Z}) \models\Si$ if and only there exists an $n\in \N$ such that $\flat(\m{Z}_p) \models\Si$ for every prime $p>n$, by Proposition~\ref{p:ZtoZp}. The claim therefore follows directly using the fact that $\Si$ is a basis for $\EqT(\cls{K})$.
\end{proof}

\begin{mthm}\label{t:basis}
There is no non-trivial class of idempotent semifields that is finitely based.
\end{mthm} 
\begin{proof} 
Let $\cls{K}$ be any non-trivial class of idempotent semifields. Consider first any $n\in\N^{>0}$. Since the $\lang_s$-inequation $x\leq\e\jn x^n$ is satisfied by all $\ell$-groups, it belongs to $\EqT(\cls{K})$. On the other hand,  $\m{Z}_n \not\models\set{x^n\eq\e} \qearrow x\eq \e$ yields $\flat(\m{Z}_n) \not\models x \leq \e\jn x^n$, by Lemma~\ref{l:flat-group}, so $\flat(\m{Z}_n) \not\models \EqT(\cls{K})$. Hence, by Corollary~\ref{c:fin-base}, to show that $\cls{K}$ is not finitely based, it suffices to prove that $\flat(\m{Z}) \models \EqT(\cls{K})$. Moreover, by Remark~\ref{r:simplesuffice}, it suffices to show that $\flat(\m{Z})$ satisfies every simple $\lang_s$-inequation in $\EqT(\cls{K})$, recalling that, by Remark~\ref{r:regular}, since $\cls{K}$ is non-trivial, these simple $\lang_s$-inequations are all left-regular.

Let $\eps = (s\leq t_1\jn\cdots\jn t_n)$ be any left-regular simple $\lang_s$-inequation and suppose contrapositively that $\flat(\m{Z}) \not\models \eps$. Then $\m{Z} \not\models Q(\eps)$, by Lemma~\ref{l:flat-group}, i.e., there exists a homomorphism $\f\colon\m{Tm}(\lang_m)\to\m{Z}$ such that $\f(t_1) =\cdots = \f(t_n)$ and $\f(t_1) \neq \f(s)$. Moreover, we can assume that $\f(t_1) < \f(s)$ in the standard order on $\Z$, and let $\hat{\f}\colon\m{Tm}(\lang_s)\to\tuple{\Z,\max,+,0}$ be the homomorphism defined by setting $\hat{\f}(x):=\f(x)$ for each variable~$x$. Then $\hat{\f}(s)=\f(s) > \f(t_1)= \hat{\f}(t_1)= \max\{ \hat{\f}(t_1), \dots, \hat{\f}(t_n)\}=\hat{\f}(t_1\jn\cdots\jn t_n)$. So $\tuple{\Z,\max, +,0} \not\models \eps$. Now consider any non-trivial $\m{A}\in\cls{K}$ and $a\in A$ with $a>\e$. The homomorphism induced by mapping $1$ to $a$ embeds $\tuple{\Z,\max, +,0}$ into $\m{A}$, so also $\m{A}\not\models \eps$ and $\eps\not\in\EqT(\cls{K})$.
\end{proof}

Let us remark finally that the approach followed in this section to establish Theorem~\ref{t:basis} extends to a broader class of idempotent semirings. More precisely, there is no finitely based class of idempotent semirings $\cls{K}$  such that $\flat(\m{Z}) \models \EqT(\cls{K})$ and $\cls{K}\models x\leq\e\jn x^n$ for every $n\in \N^{>0}$. In particular,  the class of totally ordered idempotent semirings is not finitely based, which follows also from~\cite[Theorem~48]{Aceto2003}.


\section{The number of equational theories}\label{s:cardinality-problem}

For any countable signature $\lang$ and class of $\lang$-algebras $\cls{K}$, there can be at most continuum-many equational theories of subclasses of $\cls{K}$. In particular, although there are just two equational theories of classes of commutative idempotent semifields --- $\EqT(\tuple{\Z,\max,+,0})$ and the set of all $\lang_s$-equations --- the maximum number is attained in the setting of commutative idempotent semirings.

\begin{proposition}\label{p:cont-subvar}
There are continuum-many equational theories of classes of commutative idempotent semirings.
\end{proposition}
\begin{proof}
For any set of primes $P$, let $\Si_P$ denote the equational theory of the class of commutative idempotent semirings satisfying $x \leq \e\jn x^p$ for all $p\in P$. We show that $\Si_P \neq \Si_Q$ for any two distinct sets of primes $P$ and $Q$, and hence that there are continuum-many such equational theories. Consider, without loss of generality, $p\in P{\setminus}Q$. Then $x \leq \e\jn x^p\in \Si_P$. But also, since $\m{Z}_p \not\models\set{\e\eq x^p}\qearrow\e\eq x$ and $\m{Z}_p \models \{\e \eq x^q\} \qearrow \e \eq x$ for all $q\in Q$, it follows from Lemma~\ref{l:flat-group} that $\flat(\m{Z}_p) \not\models x \leq \e\jn x^p$ and $\flat(\m{Z}_p) \models x \leq \e\jn x^q$  for all $q\in Q$. So $x \leq \e\jn x^p\not\in \Si_Q$. 
\end{proof}

To prove that the maximum number of equational theories is attained also in the setting of idempotent semifields (Theorem~\ref{t:continuum}), we make use of a corresponding result for a certain class of $\ell$-groups. Note first that, unlike idempotent semifields, $\ell$-groups form a {\em variety}: a class of algebras of the same signature that is closed under taking homomorphic images, subalgebras, and direct products --- or, equivalently, by Birkhoff's theorem, an equational class. Equational theories of classes of $\ell$-groups are therefore in one-to-one correspondence with varieties of $\ell$-groups. In particular, equational theories of classes of totally ordered groups correspond to varieties of $\ell$-groups that are {\em representable}, that is, subalgebras of direct products of totally ordered groups. For further details and references, we refer to~\cite{AF88}.

Let $\lang_g$ and $\lang_\ell$ be the signatures of groups and $\ell$-groups, respectively. A crucial role in our proof of Theorem~\ref{t:continuum} will be played by the following result, recalling that a variety $\cls{V}_1$ is defined relative to a variety $\cls{V}_2$ by a set of equations $\Si$ if $\cls{V}_1$ consists of all the members of $\cls{V}_2$ that satisfy $\Si$:

\begin{theorem}[{\cite[Theorem~1]{KM77}}] \label{t:uncountablereplgroups}
There are continuum-many varieties defined relative to the variety of representable $\ell$-groups by a set of $\lang_g$-equations.
\end{theorem}

\noindent
To prove that there are continuum-many equational theories of classes of idempotent semifields, it therefore suffices, by Theorem~\ref{t:uncountablereplgroups}, to show that any two varieties defined relative to the variety of representable $\ell$-groups by sets of $\lang_g$-equations can be distinguished by an $\lang_s$-equation. As we show below, such equations can be obtained by `eliminating inverses' from $\lang_\ell$-inequations.

Let us say that a variety $\cls{V}$ of $\ell$-groups has the {\em product-splitting property} if for any $\lang_{\ell}$-terms $s,t,u$ and any variable $y$ that does not occur in $s,t,u$,
\[
\cls{V}\models\e \le u\jn st\iff\cls{V}\models\e \le u\jn sy\jn \iv{y}t.
\]

\begin{lemma}\label{l:repdens}
Let $\cls{V}$ be any variety that is defined relative to the variety of representable $\ell$-groups by a set of $\lang_g$-equations. Then $\cls{V}$ has the product-splitting property.
\end{lemma}
\begin{proof}
Consider any $\lang_{\ell}$-terms $s,t,u$ and any variable $y$ that does not occur in $s,t,u$. Suppose first that $\cls{V}\models\e \le u\jn st$. Every $\ell$-group satisfies the  $\lang_s$-quasiequation $\set{\e \le x\jn yz}\qearrow\e \le x\jn y\jn z$ (cf.~\cite[Lemma 3.3]{GM16}), so also $\cls{V}\models\e \le u\jn sy\jn \iv{y}t$. 

Now suppose that $\cls{V}\not\models\e \le u\jn st$. Since $\cls{V}$ is a variety of representable $\ell$-groups, there exists a totally ordered group $\m{L}\in\cls{V}$ and homomorphism $\f\colon \m{Tm}(\lang_{\ell}) \to \m{L}$ satisfying $\e > \f(u\jn st) =  \f(u)\jn \f(s)\f(t)$ and therefore also $\e > \f(u)$ and $\iv{\f(s)} > \f(t)$. Let $\m{M}$ be the totally ordered group consisting of the direct product of the group reducts of $\m{L}$ and $\m{Q} = \tuple{\Q,\min,\max,+,-,0}$ equipped with the lexicographic order on $L\times\Q$. Clearly, $\p\colon \m{L} \to \m{M};\:a\mapsto\tuple{a,0}$ is an embedding of $\m{L}$ into~$\m{M}$. Note that $\m{Q}\in\cls{V}$, since the variety of Abelian $\ell$-groups is the unique atom in the subvariety lattice of the variety of $\ell$-groups (see \cite{AF88}). Moreover, since the group reduct of $\m{M}$ is the direct product of the group reducts of $\m{L}$ and $\m{Q}$, any $\lang_g$-equation satisfied by $\m{L}$ and $\m{Q}$ --- in particular, those defining $\cls{V}$ relative to  the variety of representable $\ell$-groups --- is satisfied by $\m{M}$. So $\m{M}\in\cls{V}$. 

Now let $\hat{\f} \colon \m{Tm}(\lang_\ell) \to \m{M}$ be the homomorphism defined by setting $\hat{\f}(y) := \tuple{\f(t),1}$ and $\hat{\f}(x) =\p\f(x) = \tuple{\f(x),0}$ for each variable $x \neq y$. Since $y$ does not occur in $s,t,u$, clearly $\hat{\f}(u) = \tuple{\f(u), 0}$, $\iv{\hat{\f}(s)} = \tuple{\iv{\f(s)}, 0}$, and $\hat{\f}(t) = \tuple{\f(t), 0}$. So $\hat{\f}(\e)=\tuple{\e,0} > \hat{\f}(u)$ and $\iv{\hat{\f}(s)}> \hat{\f}(y) > \hat{\f}(t)$, yielding $\hat{\f}(\e)=\tuple{\e,0}> \hat{\f}(s)\hat{\f}(y) = \hat{\f}(sy)$ and $\hat{\f}(\e)=\tuple{\e,0} > \iv{\hat{\f}(y)}\hat{\f}(t)  = \hat{\f}(\iv{y}t)$, and then, combining these inequalities, $\hat{\f}(\e)= \tuple{\e,0}> \hat{\f}(u\jn sy\jn \iv{y}t)$. Hence $\cls{V} \not\models \e \leq u\jn sy\jn \iv{y}t$.
\end{proof}

\begin{remark}
The proof of Lemma~\ref{l:repdens} establishes that every variety $\cls{V}$ defined relative to the variety of representable $\ell$-groups by a set of $\lang_g$-equations is {\em densifiable}, that is, every totally ordered member of $\cls{V}$ embeds into a dense totally ordered member of $\cls{V}$. Clearly, every densifiable variety of representable $\ell$-groups has the product-splitting property; indeed, densifiability is equivalent to  a slightly stronger version of this property in the broader setting of semilinear residuated lattices (see~\cite{MPT23} for details).
\end{remark}

Observe next that if $\cls{V}$ is a variety of $\ell$-groups that has the product-splitting property, then for any  $\lang_{\ell}$-terms $r,s,t,u,v$ and variable $y$ that does not occur in $r,s,t,u,v$,
\begin{align*}
\cls{V} \models u \le v \jn s\iv{r}t\:
\iff\: & \cls{V} \models\e \le v\iv{u} \jn s\iv{r}t\iv{u}\\
\iff\: & \cls{V} \models\e \le v\iv{u} \jn sy \jn \iv{y}\iv{r}t\iv{u}\\
\iff\: & \cls{V} \models ryu \le ryv \jn rysyu \jn t.
\end{align*}
Let us call an  $\lang_\ell$-inequation $\eps$ {\em basic} if it is of the form $s\le t_1\jn\cdots\jn t_n$ for an $\lang_m$-term $s$ and $\lang_g$-terms $t_1,\dots,t_n$.  For any basic $\lang_\ell$-inequation $\eps$, the $\lang_s$-inequation $\eps^\star$ is defined recursively as follows:\smallskip

\begin{enumerate}

\item[$\bullet$]	If  $\eps=(s\le t_1\jn\cdots\jn t_n)$ is an $\lang_s$-inequation, let $\eps^\star:=\eps$; otherwise, let $i\in\{1,\dots,n\}$ be minimal such that $t_i=u\iv{x}v$ for some $\lang_m$-term $u$, and let
\[
\eps^\star:=(xys \le xyt_1\jn\cdots\jn xyt_{i-1}\jn xyuxs\jn v\jn xyt_{i+1} \jn\cdots\jn xyt_n)^\star.
\] 

\end{enumerate}

An induction on the number of occurrences in $\eps$ of the inverse operation symbol establishes:

\begin{proposition}\label{p:inversefree}
Let $\cls{V}$ be any variety of $\ell$-groups that has the product-splitting property. Then $\cls{V}\models\eps\iff\cls{V}\models\eps^\star$ 
for any basic $\lang_\ell$-inequation $\eps$.
\end{proposition}

\begin{mthm}\label{t:continuum}
There are continuum-many equational theories of classes of idempotent semifields.
\end{mthm}
\begin{proof}
By Theorem~\ref{t:uncountablereplgroups}, there are continuum-many varieties defined relative to the variety of representable $\ell$-groups by $\lang_g$-equations, and, by Lemma~\ref{l:repdens}, all these varieties have the product-splitting property. For each such variety $\cls{V}$, let $\Si_\cls{V}$ be the equational theory of the class of the idempotent semiring reducts of the members of $\cls{V}$. Consider now any two distinct varieties $\cls{V}_1$ and $\cls{V}_2$  defined relative to the variety of representable $\ell$-groups by sets of $\lang_g$-equations. Without loss of generality, there exists a basic $\lang_\ell$-inequation $\eps$ such that $\cls{V}_1\models\eps$ and $\cls{V}_2\not\models\eps$. But then also $\cls{V}_1\models\eps^\star$ and $\cls{V}_2\not\models\eps^\star$, by Proposition~\ref{p:inversefree}, and since $\eps^\star$ is an $\lang_s$-equation, $\Si_{\cls{V}_1}\neq\Si_{\cls{V}_2}$. Hence there are continuum-many equational theories of classes of idempotent semifields.
\end{proof}


\section{The complexity of the equational theory of the class of idempotent semifields}\label{s:complexity}

The equational theory of the variety of Abelian $\ell$-groups is co-NP-complete~\cite[Theorem~1]{Wei86}, but it follows from the fact that linear programming  is solvable in polynomial time that checking if a basic $\lang_\ell$-inequation is satisfied by all Abelian $\ell$-groups belongs to P.  On the other hand, not only the equational theory of the variety $\cls{LG}$ of $\ell$-groups~\cite[Theorem~8.3]{GM16}, but also, as we show here, checking if a basic $\lang_\ell$-inequation is satisfied by $\cls{LG}$, are co-NP-complete, and hence the same is true for the equational theory of the class of idempotent semifields (Theorem~\ref{t:conp}). Indeed, we establish this result by giving a polynomial reduction of the problem of checking the satisfaction of certain $\lang_{\ell}$-inequations by $\cls{LG}$ that is known to be co-NP-hard, to the problem of checking the satisfaction of certain simple $\lang_s$-inequations by the class of idempotent semifields.

Let us say that a variety $\cls{V}$ of $\ell$-groups has the {\em meet-splitting property} if for any $\lang_{\ell}$-terms $s,t,u$ and variable $y$ that does not occur in $s,t,u$,
\[
\cls{V}\models\e \le u\jn (s\mt t) \iff\cls{V}\models\e \le u\jn sy\jn t\iv{y}.
\]

\begin{lemma}\label{l:elimLG}
The variety of $\ell$-groups has the product-splitting and meet-splitting properties.
\end{lemma}
\begin{proof}
The fact that $\cls{LG}$ has the product-splitting property is established in~{\cite[Lemma 4.1]{CGMS22}}. 

To establish the left-to-right direction of the meet-splitting property for $\cls{LG}$, it suffices to show that $\cls{LG}\models u\jn (s\mt t) \leq u\jn sy\jn t\iv{y}$; just  note that for any $\m{L}\in\cls{LG}$ and $a,b,c,d \in L$, since $\e \leq d\jn \iv{d}$, 
\[
a\jn (b\mt c) \le  a\jn (b\mt c)(d\jn \iv{d}) = a\jn (b\mt c)d\jn (b\mt c)\iv{d} \leq a\jn bd\jn c\iv{d}.
\]
For the converse, let $s,t,u$ be any $\lang_{\ell}$-terms, let $y$ be a variable that does not occur in $s,t,u$, and suppose that $\cls{LG}\not\models\e \le u\jn (s\mt t)$. Using lattice-distributivity, we may assume that $\cls{LG}\not\models\e \le u\jn s$, the case where $\cls{LG}\not\models\e \le u\jn t$ being very similar. An $\lang_{\ell}$-equation is satisfied by $\cls{LG}$ if and only if it is satisfied by the $\ell$-group $\m{Aut}(\tuple{\R,\le})$ consisting of the group of order-preserving bijections of the totally ordered set $\tuple{\R,\le}$ equipped with the pointwise lattice-order~\cite[Corollary to Lemma~3]{Hol76}. Hence $\m{Aut}(\tuple{\R,\le})\not\models\e \le u\jn s$, and there exists a homomorphism $\f \colon \m{Tm}(\lang_\ell) \to \m{Aut}(\tuple{\R,\le})$ and $q\in\R$ such that $(q)\f_u<q$ and $(q)\f_s<q$, where we assume that order-preserving bijections act on $\tuple{\R,\le}$ from the right, and write $\f_v$ for $\f(v)$. 

We obtain a homomorphism $\hat{\f}\colon \m{Tm}(\lang_\ell)\to\m{Aut}(\tuple{\R,\le})$ by defining $\hat{\f}_x := \f_x$ for every variable $x \neq y$ and defining   $\hat{\f}_y$ such that $(q)\f_s\hat{\f}_y=(q)\f_s<q$ and $(q)\f_t<(q)\hat{\f}_y$. Note that such a definition of $\hat{\f}_y$ is possible because $(q)\hat{\f}_y$ can be chosen to be arbitrarily large and any partial order-preserving injective map on $\tuple{\R,\le}$ extends linearly to a member of $\m{Aut}(\tuple{\R,\le})$. It follows, since $y$ does not occur in $s,t,u$, that  $(q)\hat{\f}_u=(q)\f_u<q$, $(q)\hat{\f}_s\hat{\f}_y=(q)\f_s\hat{\f}_y=(q)\f_s<q$, and $(q)\hat{\f}_t=(q)\f_t<(q)\hat{\f}_y$. Hence $\cls{LG}\not\models\e \le u\jn sy\jn t\iv{y}$. 
\end{proof}

\begin{remark}
No non-trivial proper subvariety of $\cls{LG}$ has the meet-splitting property. It follows easily from~{\cite[Example 13]{McC82}} that the $\lang_\ell$-equation $\e\le(x\jn\e)^2\iv{z}\jn\iv{(x\jn\e)}z\jn\iv{(x\jn\e)}$ is satisfied by every proper subvariety of $\cls{LG}$; indeed, it axiomatizes the variety of normal-valued $\ell$-groups, the unique co-atom in the subvariety lattice of $\ell$-groups, relative to $\cls{LG}$. Hence, if a proper subvariety of $\cls{LG}$ has the meet-splitting property, it satisfies also $\e\le\iv{(x\jn\e)}$ and is trivial.
\end{remark}

\begin{proposition}\label{p:basic-coNP}
The problem of checking if a basic $\lang_\ell$-inequation is satisfied by the variety of $\ell$-groups is co-NP-complete.
\end{proposition}
\begin{proof}
Since checking if an $\lang_\ell$-inequation is satisfied by $\cls{LG}$ belongs to co-NP~\cite[Theorem~8.3]{GM16}, it suffices to present a polynomial time algorithm that given input with an $\lang_\ell$-equation $\eps$ of the form
\[
\bigwedge_{i\in I}\bigvee_{j\in J_i} s_{ij}\le\bigvee_{k\in K}\bigwedge_{l\in L_k} t_{kl}\jn u_{kl},
\]
where each $s_{ij}$, $t_{kl}$, and $u_{kl}$  is an $\lang_g$-term, outputs a basic $\lang_\ell$-inequation $\de$ that has size polynomial in the size of $\eps$ such that $\cls{LG}\models\eps\iff\cls{LG}\models\de$. This is a consequence of the fact that the problem of checking validity in $\cls{LG}$ of such $\lang_\ell$-equations is co-NP-hard, since this is the case even for distributive lattice equations of this form, where each $s_{ij}$, $t_{kl}$, and $u_{kl}$ is a variable~\cite[Corollary 2.7]{HRB87}.

First, we do a little preprocessing, using the fact that $\cls{LG}$ has the product-splitting property for the third equivalence:
\begin{align*}
\cls{LG}\models\eps
& \iff \cls{LG}\models\e\le(\bigvee_{k\in K}\bigwedge_{l\in L_k} t_{kl}\jn u_{kl})\iv{(\bigwedge_{i\in I}\bigvee_{j\in J_i} s_{ij})}\\
&  \iff \cls{LG}\models\e\le(\bigvee_{k\in K}\bigwedge_{l\in L_k} t_{kl}\jn u_{kl})(\bigvee_{i\in I}\bigwedge_{j\in J_i} \iv{s}_{ij})\\
& \iff  \cls{LG}\models\e\le(\bigvee_{k\in K}\bigwedge_{l\in L_k} t_{kl}\jn u_{kl})y\jn \iv{y}(\bigvee_{i\in I}\bigwedge_{j\in J_i} \iv{s}_{ij})\\
& \iff  \cls{LG}\models\e\le(\bigvee_{k\in K}\bigwedge_{l\in L_k} t_{kl}y\jn u_{kl}y)\jn (\bigvee_{i\in I}\bigwedge_{j\in J_i} \iv{y}\iv{s}_{ij}).
\end{align*}
Hence we may assume without loss of generality that $\eps$ is an $\lang_\ell$-equation of the form $\e\le u_1\jn\cdots\jn u_n$, where each $u_i$ is a meet of binary joins of $\lang_g$-terms, and let $S$ be the size of $\eps$, so that at most $S$ $\lang_g$-terms and at most $S$ meets occur in $\eps$. 

If $\eps$ contains no meets, it is a basic $\lang_\ell$-inequation of the required form. Otherwise, suppose that $u_1= s_1\mt\cdots\mt s_m$ such that  $s_i = s_{i1}\jn s_{i2}$ and let $u:=u_2\jn\cdots\jn u_n$. By choosing distinct variables $y_1,\dots,y_m$ that do not occur in $\eps$ and using the meet-splitting property repeatedly, 
\begin{align*}
\cls{LG}\models\eps
& \iff \cls{LG}\models\e\le u\jn (s_1\mt\cdots\mt s_m)\\
& \iff \cls{LG}\models\e\le u\jn s_1y_1\jn (s_2\mt\cdots\mt s_m)\iv{y}_1\\
& \iff \cls{LG}\models\e\le u\jn s_1y_1\jn (s_2\iv{y}_1\mt\cdots\mt s_m\iv{y}_1)\\
& \iff \qquad\vdots\\
& \iff \cls{LG}\models\e\le u\jn s_1y_1\jn s_2\iv{y}_1y_2\jn\cdots\jn s_m\iv{y}_1\cdots\iv{y}_{m-1}y_m\\
&\iff \cls{LG}\models\e\le u\jn s_{11}y_1\jn s_{12}y_1\jn (\bigvee_{i=2}^m s_{i1}\iv{y}_{1}\cdots \iv{y}_{i-1}y_i\jn s_{i2}\iv{y}_{1}\cdots \iv{y}_{i-1}y_i).
\end{align*}
Let $\eps'$ be the $\lang_\ell$-equation $\e\le u\jn s_{11}y_1\jn s_{12}y_1\jn (\bigvee_{i=2}^m s_{i1}\iv{y}_{1}\cdots \iv{y}_{i-1}y_i\jn s_{i2}\iv{y}_{1}\cdots \iv{y}_{i-1}y_i)$, observing that $\eps'$ has the same number of $\lang_g$-terms as $\eps$, and that each of these $\lang_g$-terms has size at most $2S$. Hence, repeating this procedure for $u_2,\dots,u_n$, we obtain in polynomial time a basic $\lang_\ell$-inequation $\de$ of size at most $2S^2$ such that $\cls{LG}\models\eps\iff\cls{LG}\models\de$.
\end{proof}

\begin{proposition}\label{p:simple-coNP}
The problem of checking if a simple $\lang_s$-inequation is satisfied by the class of semifields (or, equivalently, by the variety of $\ell$-groups) is co-NP-complete.
\end{proposition}

\begin{proof}
Recall that $\cls{LG}$ has the product-splitting property, by Lemma~\ref{l:elimLG}. Hence, it suffices, by Propositions~\ref{p:basic-coNP} and~\ref{p:inversefree}, to present a polynomial time algorithm that given a basic $\lang_\ell$-inequation $\eps$ of the form $s \leq t_1\jn\cdots\jn t_k$ as input produces the simple $\lang_s$-inequation $\eps^\star$ as output, where the size of $\eps^\star$ is polynomial in the size of $\eps$. Let $S$ be the size of $\eps$ and recall the recursive definition of $\eps^\star$ from Section~\ref{s:cardinality-problem}. The number of inverses in $\eps$ is bounded by $S$ and decreases in every step of the recursive definition, so the algorithm stops after at most $S$ steps and yields the $\lang_s$-inequation $\eps^\star$.  Moreover, each step increases the length of the $\lang_m$-term on the left of the basic $\lang_\ell$-inequation by $2$ and increases by $1$ the number of $\lang_g$-terms on the right. The length of the $\lang_m$-term on the left is therefore at most $3S$ and the number of $\lang_g$-terms is at most $2S$ in every step. It follows that in every step the size of the $\lang_s$-inequation is increased by at most $3S+2 \cdot 2S = 7S$ and hence $\eps^\star$ is of size at most $7S^2 + S$. It is also clear that $\eps^\star$ can be computed from $\eps$ in polynomial time.
\end{proof}

As an immediate consequence of Proposition~\ref{p:simple-coNP} we obtain the main result of this section.

\begin{mthm}\label{t:conp}
The equational theory of the class of idempotent semifields is co-NP-complete.
\end{mthm}

We conclude this section by using Theorem~\ref{t:conp} and a correspondence established in~\cite{CM19} to prove complexity results also for the existence of right orders on free groups and monoids satisfying finitely many constraints. Recall first that a {\em right order} on a monoid $\m{M}$ is a total order $\le$ on $M$ such that $a \le b\:\Longrightarrow\:ac \le bc$ for any $a,b,c\in{M}$. For a set $X$ we let $\m{F}_g(X)$ and $\m{F}_m(X)$ denote the free group and free monoid with generators in $X$, respectively, assuming for convenience that ${\rm F}_g(X)\subseteq{\rm Tm}(\lang_g)$ and  ${\rm F}_m(X)\subseteq{\rm Tm}(\lang_m)$.

\begin{theorem}[{\cite[Theorem~2]{CM19}}]\label{t:rord-equiv}
For any set $X$ and $s_1,\ldots,s_n\in{\rm F}_g(X)$, there exists a right order $\le$ on $\m{F}_g(X)$ satisfying $\e<s_1,\dots,\e<s_n$ if and only if  $\cls{LG}\not\models\e \le s_1\jn\cdots\jn s_n$. 
\end{theorem}

\begin{corollary}\label{c:conporder}
The problem of checking for a set $X$ with $\lvert X \rvert \geq 2$ and $s_1,\ldots,s_n\in{\rm F}_g(X)$ if there exists a right order $\le$ on $\m{F}_g(X)$ satisfying $\e<s_1,\dots,\e<s_n$ is NP-complete.
\end{corollary}
\begin{proof}
Observe first that when $X$ is an infinite set, the claim follows directly from Proposition~\ref{p:basic-coNP} and Theorem~\ref{t:rord-equiv}. To establish the claim in full generality, it suffices, by Theorem~\ref{t:rord-equiv}, to consider the case where $\lvert X \rvert =2$. Let $X:=\set{x_1,x_2}$ and $Y:= \set{y_n \mid n\in \N}$. It is well-known that $\m{F}_g(Y)$ is isomorphic to the commutator subgroup $\m{G}$ of $\m{F}_g(X)$ generated by elements of the form $[x_1^k,x_2^l]$ with $k,l \in\Z{\setminus}\{0\}$~(see, e.g.,~\cite[Theorem~11.48]{Rot94}). In particular, for any bijection $\pi\colon(\Z{\setminus}\{0\})^2\to\N$, we can define an isomorphism $\f\colon\m{F}_g(Y)\to\m{G}$ such that $\f(y_{\pi(\tuple{k,l})}) := [x_1^k,x_2^l]$ for $k,l \in \Z {\setminus} \{ 0 \}$.

Moreover, since $\m{F}_g(X)/\m{G} \cong \tuple{\Z^2, +,-,\tuple{0,0}}$, there exists a right order $\preceq$ on the group $\m{F}_g(X)/\m{G}$. Hence, for any right order $\leq$ on $\m{G}$, we obtain a right order $\leq^\ast$ on $\m{F}_g(X)$ that extends $\leq$ by setting 
\[
s \leq^\ast t \defiff Gs \prec Gt \,\text{ or }\, (Gs = Gt\, \text{ and }\, s \leq t).
\]
It follows that there exists a right order on $\m{F}_g(Y)$ satisfying $\e < s_1, \dots, \e < s_n$, for some given $s_1,\ldots,s_n\in{\rm F}_g(Y)$,  if and only if there exists a right order on $\m{F}_g(X)$ satisfying $\e < \f(s_1), \dots, \e < \f(s_n)$.  Finally, note that $\pi$ can be chosen such that  $\f(s_i)$ is computable in polynomial time from $s_i$ for each $i\in\{1,\dots,n\}$, and its size is polynomial in the sum of the sizes of $s_1,\dots, s_n$.
\end{proof}

By \cite[Corollary~3.4]{CGMS22}, every right order on $\m{F}_m(X)$ extends to a right order on $\m{F}_g(X)$. Hence it follows from Theorem~\ref{t:rord-equiv} that $\cls{LG} \not\models s \leq t_1\jn\cdots\jn t_n$, where $s, t_1,\dots, t_n  \in {\rm F}_m(X)$, if and only if there exists a right order $\leq$ on $\m{F}_m(X)$ with $s < t_1, \dots, s < t_n$. Proposition~\ref{p:simple-coNP} therefore implies that the problem of checking for any $s,t_1, \dots, t_n\in{\rm F}_m(\omega)$ if there exists a right order $\leq$ on $\m{F}_m(\omega)$ satisfying $s<t_1,\dots,s<t_n$ is NP-complete. Therefore:

\begin{corollary}\label{c:freemonoid2}
The problem of checking for any $s_1,\dots,s_n,t_1, \dots, t_n\in{\rm F}_m(\omega)$ if there exists a right order $\leq$ on $\m{F}_m(\omega)$ satisfying $s_1<t_1,\dots,s_n<t_n$ is NP-complete.
\end{corollary}

Note that it does not follow directly from the previous results that Corollary~\ref{c:freemonoid2} is true when $\m{F}_m(\omega)$ is replaced by $\m{F}_m(X)$ for a set $X$ with $2\le\lvert X \rvert <\omega$, the main obstacle being that the translation $\f$ in the proof of Corollary~\ref{c:conporder} introduces new inverses, while the elimination of inverses in the proof of Proposition~\ref{p:simple-coNP} introduces new variables. 


\section{Related structures}\label{s:signature}

In this final section, we show that the results of the previous sections extend in many cases to other classes of algebraic structures that are closely related to idempotent semifields and $\ell$-groups.

Let us remark first that expanding any idempotent semifield with the lattice meet operation produces a {\em distributive $\ell$-monoid}: an algebraic structure $\tuple{L,\mt,\jn,\pd,\e}$ such that $\tuple{L,\pd,\e}$ is a monoid; $\tuple{L, \mt,\jn}$ is a distributive lattice; and multiplication distributes over binary meets and joins. It follows directly from Theorem~\ref{t:continuum} that there are continuum-many equational theories of classes of distributive $\ell$-monoids with an idempotent semifield reduct, and from Theorem~\ref{t:conp}, that the equational theory of the class of distributive $\ell$-monoids with an idempotent semifield reduct is co-NP-complete.  Although not all distributive $\ell$-monoids are meet-expansions of idempotent semifields (equivalently, inverse-free reducts of $\ell$-groups), the equational theories of the classes of distributive $\ell$-monoids and meet-expansions of idempotent semifields (equivalently, inverse-free reducts of $\ell$-groups) coincide~\cite[Theorem~2.9]{CGMS22}. Hence the equational theory of the class of meet-expansions of idempotent semifields is finitely based and an analogue of Theorem~\ref{t:basis} does not hold in this setting. Note, however, that even though the equational theory of the variety of Abelian $\ell$-groups is finitely based, this is not the case for the class of their inverse-free reducts~\cite[Theorem 2]{Rep83}. 

Recall next that idempotent semifields are sometimes formulated in the literature without the neutral element $\e$ in the signature, that is, as {\em $\e$-free reducts} of idempotent semifields as defined in this paper.  Observe, however, that a simple $\lang_s$-inequation $s\le t_1\jn\cdots\jn s_n$ is satisfied by an idempotent semifield $\m{S}$ if and only if its   $\e$-free reduct satisfies $(xs)^\circ\le (xt_1)^\circ\jn\cdots\jn (xs_n)^\circ$, where $x$ is any variable not occurring in $s,t_1,\dots,t_n$, and $v^\circ$ is obtained by removing all occurrences of $\e$ from an $\lang_s$-term $v$. Hence, we obtain easily the following analogues of Theorems~\ref{t:basis},~\ref{t:continuum}, and~\ref{t:conp}: there is no non-trivial class of $\e$-free reducts of idempotent semifields that is finitely based, there are continuum-many equational theories of classes of $\e$-free reducts of idempotent semifields, and the equational theory of the class of $\e$-free reducts of idempotent semifields is co-NP-complete.

Finally, recall that idempotent semirings are also sometimes formulated in the literature with both $\e$ and a constant symbol $0$ interpreted as the neutral element of $\jn$. We show here that analogues of our Theorems~\ref{t:basis},~\ref{t:continuum}, and~\ref{t:conp} also hold in this setting, using similar methods to~\cite[Section 4]{Aceto2003}. Let us call an algebraic structure $\tuple{F,\jn,\pd,\e,0}$ an {\em idempotent $0$-semiring} if $\tuple{F,\jn,\pd,\e}$ is an idempotent semiring with least element $0$, and an  \emph{idempotent 0-semifield} if, additionally, $\tuple{F{\setminus}\{0\},\pd,\e}$ is the monoid reduct of a group. Clearly, if $\m{F}=\tuple{F,\jn,\pd,\e,0}$ is an idempotent $0$-semifield, then $\m{F}^*:=\tuple{F{\setminus}\{0\},\jn,\pd,\e}$ is an idempotent semifield. Conversely, given any idempotent semiring $\m{S}=\tuple{S,\jn,\pd,\e}$ and element $0\not\in S$, the algebraic structure $\m{S}_0:=\tuple{S\cup\{0\},\jn,\pd,\e,0}$ satisfying $0\le a$ and $0\pd a = a\pd 0=0$ for all $a\in S\cup\{0\}$, is an idempotent $0$-semiring. In particular, if $\m{S}$ is an idempotent semifield, then $\m{S}_0$ is an idempotent $0$-semifield. Moreover, $(\m{F}^*)_0=\m{F}$ for each  idempotent $0$-semifield $\m{F}$, and $(\m{S}_0)^*=\m{S}$ for each idempotent semifield $\m{S}$. Hence there is a one-to-one correspondence between classes of idempotent $0$-semifields and classes of  idempotent semifields implemented by the maps $\cls{K}\mapsto\{\m{F}^*\mid\m{F}\in\cls{K}\}$ and $\cls{K}\mapsto \{\m{S}_0 \mid \m{S} \in \cls{K} \}$.

\begin{remark}\label{r:equivterm}
There is a linear time algorithm that given any term $t$ in the signature of idempotent $0$-semirings, produces a smaller term $t'$ that is either $0$ or an $\lang_s$-term such that $t\eq t'$ is satisfied by all idempotent $0$-semirings that satisfy the absorption laws $x \cdot 0 \eq 0$ and $0 \cdot x \eq 0$ (in particular, all idempotent $0$-semifields).
\end{remark}

Let us call a simple $\lang_s$-inequation $s\le t_1\jn\cdots\jn t_n$ \emph{right-regular} if every variable occurring in $t_1,\dots, t_n$ occurs in $s$. 

\begin{lemma}\label{l:correspondence}
Let $\m{S}$ be any idempotent semiring. Then a right-regular simple $\lang_s$-inequation is satisfied by $\m{S}$ if and only if it is satisfied by $\m{S}_0$.
\end{lemma}
\begin{proof}
The right-to-left direction follows from the fact that $\m{S}$ is a subreduct of $\m{S}_0$. For the left-to-right direction, it suffices to observe that a right-regular simple $\lang_s$-inequation is satisfied by $\m{S}_0$ under any assignment that maps one of the variables occurring in it to $0$.
\end{proof}

\begin{lemma}\label{l:semifield-right-regular}
Let $\m{S}$ be any non-trivial idempotent semifield. Then for any simple $\lang_s$-inequation $\eps = (s \le t_1\jn\cdots\jn t_n)$ satisfied by $\m{S}$, there exists a subset $\{t_{i_1}, \dots, t_{i_k} \} \subseteq \{t_1,\dots, t_n \}$ such that the simple $\lang_s$-inequation $s \le t_{i_1}\jn\cdots\jn t_{i_k}$ is right-regular and $\m{S} \models s \le t_{i_1}\jn\cdots\jn t_{i_k}$.
\end{lemma}
\begin{proof}
Note first that, since $\m{S}$ is non-trivial, there cannot be a variable $x$ that occurs in each of $t_1,\dots,t_n$ but not in $s$; otherwise, we could assign $x$ to some element $a<\e$ in $S$ and all other variables to $\e$ to arrive at a contradiction. Hence it suffices to prove that (assuming a suitable permutation of $t_1,\dots,t_n$) if a variable $x$ occurs in each of $t_{k+1},\dots,t_n$, but not in $s,t_1,\dots,t_k$ for some $k<n$, then $\m{S}$ satisfies $s \le t_1\jn\cdots\jn t_k$. In this way we can inductively eliminate all the terms on the right that contain a variable that does not occur in $s$ and obtain a right-regular simple $\lang_s$-inequation.

Suppose contrapositively that $\m{S}\not\models s \le t_1\jn\cdots\jn t_k$, that is, there exists a homomorphism $\f\colon\m{Tm}(\lang_s)\to\m{S}$ such that $\f(s)\not\le \f(t_1)\jn\cdots\jn\f(t_k)$. We may assume that for $i\in \{k+1,\dots, n \}$, each $t_i$ is of the form $u_1xu_2x\cdots u_{l-1}xu_{l}$, where $u_1,\dots,u_l$ do not contain $x$, by inserting $\e$ where needed. 
Let $\hat{\f}\colon\m{Tm}(\lang_s)\to\m{S}$ be the homomorphism defined by setting $\hat{\f}(y):=\f(y)$ for every variable $y\neq x$ and $\hat{\f}(x)$ to be the meet of all the elements of the form $\iv{\f(u)} \mt \iv{\f(u)}\f(t_1)\iv{\f(v)}$ such that $u,v$ are subterms of $t_{k+1},\dots,t_n$ not containing $x$. Then for each $i\in \{k+1,\dots, n\}$ and $t_i = u_1 x u_2x \cdots u_{l-1}xu_l$, where $u_1,\dots,u_l$ do not contain $x$,
\begin{align*}
\hat{\f}(t_i) &= \f(u_1) \hat{\f}(x) \f(u_2)\hat{\f}(x) \cdots \f(u_{l-1}) \hat{\f}(x) \f(u_l) \\
			    &\leq \f(u_1)  \iv{\f(u_1)} \f(u_2) \iv{\f(u_2)}\ \cdots \f(u_{l-1}) \iv{\f(u_{l-1})} \f(t_1) \iv{\f(u_l)} \f(u_l) \\
			    &= \f(t_1).
\end{align*}
Hence $\hat{\f}(t_{k+1})\jn\cdots\jn\hat{\f}(t_n)\le\f(t_1)$, so $\hat{\f}(s)=\f(s)\not\le\f(t_1)\jn\cdots\jn\f(t_k)=\hat{\f}(t_1)\jn\cdots\jn\hat{\f}(t_n)$, and $\m{S}\not\models s \le t_1\jn\cdots\jn t_n$. 
\end{proof}

The equational theory of any non-empty class of idempotent $0$-semifields containing exactly two elements has a finite basis consisting of the defining equations for bounded distributive lattices with meet operation $\pd$, greatest element $\e$, and least element $0$. For convenience, let us call a class of idempotent $0$-semifields {\em non-Boolean} if at least one of its members has more than two elements. Clearly, a class $\cls{K}$ of idempotent $0$-semifields is non-Boolean if and only if $\cls{K}^*$ is non-trivial.

\begin{corollary}\label{c:0-fin-base}
Every non-Boolean class of idempotent $0$-semifields is not finitely based.
\end{corollary}
\begin{proof}
Suppose towards a contradiction that $\cls{K}$ is a finitely based non-Boolean class of idempotent $0$-semifields. Using Remark~\ref{r:equivterm} and Lemma~\ref{l:semifield-right-regular}, we may assume that $\Si\cup\{0\le x, x\cdot 0 \eq 0, 0\cdot x \eq 0\}$ is a basis for $\EqT(\cls{K})$ for some finite set of right-regular simple $\lang_s$-inequations $\Si$. 
That is, $\Si\cup\{0\le x, x\cdot 0 \eq 0, 0\cdot x \eq 0\} \subseteq\EqT(\cls{K})$ and $\EqT(\cls{K})$ is a logical consequence of $\Si\cup\{0\le x, x\cdot 0 \eq 0, 0\cdot x \eq 0\}$. We claim that $\Si$ is a basis for $\EqT(\cls{K}^*)$, contradicting Theorem~\ref{t:basis}. 
Observe first that $\cls{K}\models\eps$ if and only if $\cls{K}^*\models\eps$,  for any right-regular simple $\lang_s$-inequation $\eps$, by Lemma~\ref{l:correspondence}, so $\Si\subseteq\EqT(\cls{K}^*)\subseteq\EqT(\cls{K})$. Now suppose that some idempotent semiring $\m{A}$ satisfies $\Si$. Then $\m{A}_0$ satisfies $\Si$,  by Lemma~\ref{l:correspondence}, and $\m{A}_0$ therefore satisfies $\EqT(\cls{K})$. So $\m{A}$ satisfies $\EqT(\cls{K}^*)$. Hence $\EqT(\cls{K}^*)$ is a logical consequence of $\Si$.
\end{proof}

If $\cls{K}$ and $\cls{K'}$ are classes of idempotent semifields with distinct equational theories, then there is a simple $\lang_s$-inequation that is satisfied by one and not the other, so $\cls{K}_0$ and $\cls{K}'_0$ also have distinct equational theories, by Lemma~\ref{l:correspondence} and Lemma~\ref{l:semifield-right-regular}. Hence, by Theorem~\ref{t:continuum}:

\begin{corollary}
There are continuum-many equational theories of classes of idempotent $0$-semifields.
\end{corollary}

Finally, combining Remark~\ref{r:equivterm}, Lemma~\ref{l:correspondence}, Lemma~\ref{l:semifield-right-regular}, and Theorem~\ref{t:conp}, we obtain:

\begin{corollary}
The equational theory of the class of idempotent $0$-semifields is co-NP-complete.
\end{corollary}



\affiliationone{
   G. Metcalfe and S. Santschi\\
   Mathematical Institute, University of Bern, Sidlerstrasse 5, 3012 Bern\\ 
   Switzerland
   \email{george.metcalfe@unibe.ch\\
   simon.santschi@unibe.ch}}
   
   \end{document}